\documentclass[10pt]{amsart}
\usepackage{amsmath, amsthm, amsfonts, amssymb, graphicx, bm,verbatim,mathrsfs,siunitx}
\usepackage{hyperref}
\usepackage{stmaryrd,enumerate,tikz}
 
\theoremstyle{plain}
\newtheorem{thrm}{Theorem}[section]
\newtheorem{lmm}[thrm]{Lemma}
\newtheorem*{prncpl}{Principle}

\newcommand{\Mod}[1]{\ (\mathrm{mod}\ #1)}

\title{Digits of primes}
\author{James Maynard}

\begin{document}

\begin{abstract}
We discuss some different results on the digits of prime numbers, giving a simplified proof of weak forms of a result of Maynard and Mauduit-Rivat.
\end{abstract}



\maketitle
\section{Introduction: Counting primes}
Many problems in prime number theory can be phrased as `given a set $\mathcal{A}$ of integers, how many primes are there in $\mathcal{A}$?'. Two famous examples are whether there are infinitely many primes of the form $n^2+1$, and whether there is always a prime between two consecutive squares. Here `how many' might be asking for at least one prime, whether there are finitely many or infinitely many primes, or an asymptotic estimate for the number of primes in $\mathcal{A}$ of size depending on some parameter $x$.

Typically in analytic approaches to such questions, one tries to count the number of primes in a set $\mathcal{A}$ of a given size. Almost all our approaches rely on the following rough principle, originally due to Vinogradov but with important refinements due to many authors including Fouvry, Friedlander, Harman, Heath-Brown, Iwaniec, Linnik, Vaughan, as well as many others (see \cite{Harman} for more details).
\begin{prncpl}
Given a set of integers $\mathcal{A}\subseteq[1,x]$, you can count the number of primes in $\mathcal{A}$ if you are `good' at counting
\begin{itemize}
\item The number of elements of $\mathcal{A}$ in arithmetic progressions to reasonably large modulus (at least on average).
\item Certain bilinear sums associated with the set $\mathcal{A}$.
\end{itemize}
\end{prncpl}
Here we have been deliberately vague as to what we mean by `good', `reasonably large', or the bilinear sums, since these can vary from application to application. To give a brief indication as to why such a principle might hold, we see that by inclusion-exclusion on the smallest prime factor $P^-(n)$ of $n$ we have
\begin{align*}
\#\{p\in\mathcal{A}:p>x^{1/2}\}&=\#\{n\in\mathcal{A}:P^-(n)>x^{1/4}\}\\
&\qquad-\sum_{x^{1/4}<p\le x^{1/2}}\#\{n\in\mathcal{A}:P^-(n)=p\}.
\end{align*}
The term on the left hand side is counting primes in $\mathcal{A}$. The first term on the right hand side is a quantity that can be estimated by sieve methods, and one can obtain reasonable upper and lower bounds for such quantities if one can estimate $\mathcal{A}$ in arithmetic progressions. The final term on the right hand side is counting products of 2 or 3 primes which lie in $\mathcal{A}$ where all factors are reasonably large, and so can be viewed as a special case of bilinear sums of the form
\[
\sum_{N<n<2N}\sum_{\substack{M<m<M\\ nm\in\mathcal{A}}}\alpha_n\beta_m,
\]
for some complex coefficients $|\alpha_n|, |\beta_m|\le 1$. By iterating such identities, we can therefore obtain combinatorial decompositions for the number of primes in $\mathcal{A}$ which can be estimated by using our knowledge of $\mathcal{A}$ in arithmetic progressions or certain bilinear sums over moderately large variables whose product lies in $\mathcal{A}$.

Typically it is rather harder to estimate the bilinear sums associated to a set $\mathcal{A}$ than it is to estimate $\mathcal{A}$ in arithmetic progressions, and particularly so if $\mathcal{A}$ is sparse in the sense that it contains $O(x^{1-\epsilon})$ elements. (We fail to prove the $n^2+1$ conjecture precisely because we have no way of estimating the bilinear terms- the sparsity of the sequence means that it is insufficient to consider bilinear sums with general coefficients in this case.) Moreover, the condition $nm\in\mathcal{A}$ in the bilinear sums is difficult to handle unless $\mathcal{A}$ has some `multiplicative structure'. This is why the only sparse polynomials known to take infinitely many prime values (such as the result of Friedlander-Iwaniec \cite{FriedlanderIwaniec} on primes of the form $X^2+Y^4$) are closely associated to norm forms which have such multiplicative structure.
\section{Digital properties of primes}
The aim of this article is to discuss some results on primes in sets $\mathcal{A}$ described by their digit properties. Such sets are interesting from a technical point of view since they do not possess an obvious multiplicative structure; the digital properties of two integers $n$ and $m$ are in general not closely related to the digital properties of $nm$, since carries will destroy any simple structure. Nevertheless, we are able to count primes in various sets $\mathcal{A}$ defined in terms of the digits of integers, even in situations when the set $\mathcal{A}$ is sparse.

In particular, for our discussion we focus on the following three results:
\begin{thrm}[Mauduit-Rivat \cite{MauduitRivat}]\label{thrm:Mauduit}
Let $s_b(n)$ denote the sum of digits of $n$ in base $b$.

Then for any fixed choice of $(m,b-1)=1$, we have
\[
\#\{p<x:s_b(p)\equiv a \pmod{m}\}=\frac{(1+o(1))x}{m\log{x}}.
\]
\end{thrm}
In particular, the result of Mauduit-Rivat shows that asymptotically $50\%$ of primes have an even sum of digits, and $50\%$ of primes have an odd sum of digits when viewed in base 10.
\begin{thrm}[Bourgain \cite{Bourgain}]\label{thrm:Bourgain}
Let $c>0$ be a sufficiently small constant, and $k$ be a sufficiently large integer. Then for any index set $\mathcal{I}\subseteq\{1,\dots,k\}$ with $\#\mathcal{I}\le ck$, and for any choice of base-2 digits $\epsilon_i\in\{0,1\}$ we have
\[
\#\Bigl\{p=\sum_{i=0}^kn_i2^i:\,n_i\in\{0,1\},\,n_j=\epsilon_j\text{ for }j\in\mathcal{I}\Bigr\}=(1+o(1))\frac{2^{k-\#\mathcal{I}}}{\log{2^k}}.
\]
\end{thrm}
The result of Bourgain shows that you can prescribe a positive proportion of the binary digits (of numbers with $k$ digits), and regardless of which digits are prescribed we can still find primes with the prescribed digits.

\begin{thrm}[Maynard \cite{Maynard}]\label{thrm:Maynard} 
There are constants $0<c<C$ such that for any choice of $a_0\in \{0,\dots,b-1\}$ and for any $x>b\ge 10$, we have
\[
c\frac{x^{\log(b-1)/\log{b}}}{\log{x}}\le \#\{p<x:\,p\text{ has no base $b$ digit equal to $a_0$}\}\le C\frac{ x^{\log(b-1)/\log{b}}}{\log{x}}.
\]
\end{thrm}
In particular, this shows that there are infinitely many primes which have no digit equal to 7 in their decimal expansion.

We aim to give a unified treatment to simplified forms of these results, emphasising the properties we are using to count primes. In particular, we aim to give fairly complete proofs of Theorem \ref{thrm:Mauduit} and Theorem \ref{thrm:Maynard} in the case when the base $b$ is a sufficiently large constant, and give a sketch of Bourgain's result when working with a base $b$ which is sufficiently large. Specifically we prove the following:
\begin{thrm}[Weak Theorem \ref{thrm:Mauduit}]\label{thrm:WeakMauduit}
Let $b$ be sufficiently large. Let $s_b(n)$ denote the sum of digits of $n$ in base $b$.

Then for any fixed choice of $(m,b-1)=1$, we have
\[
\#\{p<b^k:\,s_b(p)\equiv a \pmod{m}\}=\frac{(1+o(1))b^k}{mk\log{b}}.
\]
as $k\rightarrow \infty$ through the integers.
\end{thrm}
\begin{thrm}[Weak Theorem \ref{thrm:Maynard}]\label{thrm:WeakMaynard}
Let $b$ be sufficiently large. Then for any choice of $a_0\in \{0,\dots,b-1\}$ and for any $x>b$, we have
\[
\#\{p<b^k:\,p\text{ has no base $b$ digit equal to $a_0$}\}=\frac{(\kappa_b(a_0)+o(1))(b-1)^k}{k\log{b}},
\]
as $k\rightarrow \infty$ through the integers, where
\[
\kappa_b(a_0)=\begin{cases}
\displaystyle\frac{b}{b-1},\qquad&\text{if $(a_0,b)\ne 1$,}\\
\displaystyle\frac{b(\phi(b)-1)}{(b-1)\phi(b)},&\text{if $(a_0,b)=1$.}
\end{cases}
\]
\end{thrm}
In both cases, extra ideas are required to show the full theorem which go beyond a simple refinement of the method presented here. Nevertheless, the proofs of Theorems \ref{thrm:WeakMauduit} and \ref{thrm:WeakMaynard} still contain most of the qualitatively important aspects, which also appear in Theorem \ref{thrm:Bourgain}. The overall approach has much in common with work of Drmota, Mauduit and Rivat on the sum of digits function for polynomials \cite{Drmota}.
\section{Sum of digits of primes}
Any number $n$ whose sum of digits $s_{10}(n)$ in base 10 is a multiple of 3 must itself be a multiple of 3. More generally, in base $b$
\[
n=\sum_{i}n_i b^i\equiv \sum_i n_i=s_b(n)\pmod{b-1},
\]
so any integer whose sum of digits in base $b$ has a common factor $h$ with $b-1$ must itself be a multiple of $h$. In particular, the sum of digits $s_b(p)$ in base $b$ of a prime number $p>b-1$ must necessarily be coprime to $b-1$.

After a moments thought, it appears difficult to think of any other property linking the sum of digits function with the primes. This might encourage us to guess that if we look at large numbers, the properties `being prime' and `having sum of digits equal to $a$' are roughly independent, apart from the property that $a$ must be coprime with $b-1$. In particular, we might guess that asymptotically $50\%$ of prime numbers have an even sum of digits, and $50\%$ have an odd sum of digits, or that $50\%$ of primes have sum of digits which is $1\Mod{3}$ and $50\%$ a sum of digits which is $2\Mod{3}$ (since 3 is the only prime with sum of digits $0\Mod{3}$). Theorem \ref{thrm:Mauduit} is confirming this guess.
\section{Restricted  digits of primes}
Many basic and fundamental properties of primes can be recast as problems about their digits. The problem of how many primes there are in a short interval $[x,x+x^\theta]$ is essentially equivalent to asking whether there are primes with $\lceil\log{x}/\log{10}\rceil$ digits which have a specified string of $\lceil \theta\log{x}/\log{10}\rceil$ numbers as their leading decimal digits. Similarly, asking for primes with $k$ decimal digits which end in a specific string of $\theta k$ digits is asking for primes less than $10^k$ in an arithmetic progression modulo $10^{\theta k}$. The question as to whether there are infinitely Mersenne primes is equivalent to asking whether there are infinitely many primes with no digit 0 in their base $2$ expansion.

The final decimal digit of a prime number larger than 10 must be $1$, $3$, $7$ or $9$, since any integer ending in $0$, $2$, $4$, $5$, $6$ or $8$ must be a multiple of $2$ or $5$. More generally, the final digit in base $b$ of a prime $p>b$ must be coprime to $b$. Beyond this condition, however, it is not clear that there is any simple property of individual digits of primes which is constrained in any way. We might guess that for any large set $\mathcal{A}\subseteq[X,2X]$ defined only in terms of base $b$ digital properties and containing only integers with last digit coprime to $b$, the density of primes in $\mathcal{A}$ is the same as the density of primes in the set of all integers in $[X,2X]$ which have last digit coprime to $b$. This density is $\asymp b/(\phi(b)\log{X})$, so we might guess that
\[
\#\{p\in\mathcal{A}\}\approx \frac{b}{\phi(b)\log{X}}\#\mathcal{A}
\]
in this case. Theorems \ref{thrm:Bourgain} and \ref{thrm:Maynard} confirm this heuristic for the set $\mathcal{A}$ of integers with some prescribed binary digits, or the set of integers with a digit missing in their decimal expansion. (Note that $\#\mathcal{A}\approx X^{\log{(b-1)}/\log{b}}$ for numbers with no base $b$ digit equal to $a_0$.)
\section{Fourier Analysis on digit functions}
The proofs of Theorems \ref{thrm:Mauduit}-\ref{thrm:Maynard} are Fourier-analytic in nature, and ultimately rely on the fact that many digit-related functions are very well controlled by their Fourier transform. Given a function $f:\mathbb{Z}\rightarrow\mathbb{C}$, we define the Fourier transform $\widehat{f}_x:\mathbb{R}/\mathbb{Z}\rightarrow\mathbb{C}$ of $f$ restricted to $[0,x]$ by
\[
\widehat{f}_x(\theta):=\sum_{n<x}f(n)e(n\theta).
\]
Here, and throughout the paper, $e(t):=e^{2\pi i t}$ is the complex exponential.

Our weak version of Theorem \ref{thrm:Mauduit} is based on understanding $\widehat{g}_x$ when $g(n)=e(\alpha s_b(n))$ where $\alpha\in\{0,1/m,\dots,(m-1)/m\}$ and $s_b(n)$ is the sum of digits in base $b$. In particular, writing $n=\sum_{i=0}^{k-1}n_ib^{i}$ in its base $b$ expansion, we find
\begin{align*}
\widehat{g}_{b^k}(\theta)&=\sum_{n<b^k}e(n\theta)e(\alpha s_b(n))\\
&=\sum_{0\le n_0,\dots,n_{k-1}< b}e\Bigl(\sum_{i=0}^{k-1} n_i(\alpha+b^{i}\theta)\Bigr)\\
&=\prod_{i=0}^{k-1} \Bigl(\sum_{0\le n_i<b}e(n_i(\alpha+b^{i}\theta))\Bigr)\\
&=\prod_{i=0}^{k-1} \Bigl(\frac{e(b\alpha+b^{i+1}\theta)-1}{e(\alpha+b^{i}\theta)-1}\Bigr).
\end{align*}
Thus $\widehat{g}_{b^k}$ has a product structure, which will be very convenient to work with.

For our weak version of Theorem \ref{thrm:Maynard}, we work with the Fourier transform of the indicator function $\mathbf{1}_{\mathcal{B}}$ of the set $\mathcal{B}$ of integers with no base $b$ digit equal to $a_0$. Similarly to the calculation above, we have
\begin{align*}
\widehat{\mathbf{1}}_{\mathcal{B},b^k}(\theta)&=\sum_{n<b^k}e(n\theta)\mathbf{1}_{\mathcal{B}}(n)\\
&=\sum_{\substack{0\le n_0,\dots,n_{k-1}< b\\ n_i\ne b_0}}e\Bigl(\sum_{i=0}^{k-1} n_i b^{i}\theta\Bigr)\\
&=\prod_{i=0}^{k-1} \Bigl(\frac{e(b^{i+1}\theta)-1}{e(b^{i}\theta)-1}-e(b^{i}a_0\theta)\Bigr).
\end{align*}
Again, we find $\widehat{\mathbf{1}}_{\mathcal{B},b^k}$ has a nice product structure.

\section{Hardy-Littlewood circle method}

Each of the results fundamentally relies on an application of the Hardy-Littlewood circle method. If we wish to count primes in a set $\mathcal{A}\subseteq[1,b^k)$ then by Fourier inversion on $\mathbb{Z}/b^k\mathbb{Z}$, we have
\[
\sum_{n< b^k}\Lambda(n)\mathbf{1}_{\mathcal{A}}(n)
=\frac{1}{b^k}\sum_{0\le a<b^k}\widehat{\mathbf{1}}_{\mathcal{A},b^k}\Bigl(\frac{a}{b^k}\Bigr)\widehat{\Lambda}_{b^k}\Bigl(\frac{-a}{b^k}\Bigr),\]
where $\Lambda(n)$ is the Von Mangoldt function.

Heuristically, we expect that $\widehat{\Lambda}(\theta)$ is large only if $\theta$ is close to a rational with small denominator. Similarly, we expect that if $\mathcal{A}$ is defined by base-$b$ digital properties, then $\widehat{\mathbf{1}_{\mathcal{A}}}(\theta)$ will be large only if $\theta$ is close to a real whose base $b$ expansion has specific properties (such as having many 0's in the base $b$ expansion). We might hope that there should be a large contribution to the count of primes in $\mathcal{A}$ only from those $a$ for which $\widehat{\Lambda}(-a/b^k)$ and $\widehat{1}_{\mathcal{A}}(a/b^k)$ are both simultaneously large, which would occur if $a/b^k$ is close to a rational with small denominator and has specific base $b$ properties. We would expect there to be no special features to the base $b$ expansion of a rational $a/q$ unless $q$ is a divisor of a small power of $b$. Thus we might hope that the main contribution is from those $a$ such that $a/b^k$ is close to a rational with denominator a small power of $b$.

Given the above heuristic, we split the contribution up depending on whether $a/b^k$ is close to a rational with small denominator or not. (This distinguishes between those $a$ when $\widehat{\Lambda}(a/b^k)$ is large or not.) For the $a$'s for which $a/b^k$ is close to a rational with small denominator (known as the `major arcs'), we obtain an asymptotic estimate for $\widehat{\Lambda}(-a/b^k)$ and for $\widehat{\mathbf{1}_{\mathcal{A}}}(a/b^k)$ by counting primes and elements of $\mathcal{A}$ in arithmetic progressions. For those $a$ where $a/b^k$ is not close to a rational of small denominator (the `minor arcs'), we wish to obtain a suitable upper bound for the contribution to show such terms do not contribute much to our count. This relies on a $L^\infty$ bound for $\hat{\Lambda}$ and a $L^1$ bound for $\widehat{\mathbf{1}_{\mathcal{A}}}$ for such $a$.

For counting primes with restricted digits, we apply this directly with $\mathcal{A}$ being the set $\mathcal{B}$ of integers with no digit equal to $a_0$. For the case of the sum of digits of primes, we perform a small variation to simplify things. We we see that if $\mathcal{A}$ is the set of integers $n$ with $s_b(n)\equiv a\Mod{m}$ then recalling $g(n)=e(\alpha s_b(n))$, we have
\[
\widehat{\mathbf{1}}_{\mathcal{A},x}(\theta)=\frac{1}{m}+\frac{1}{m}\sum_{\alpha\in\{1/m,\dots,(m-1)/m\}}e(-a\alpha)\widehat{g}_{x}(\theta),
\]
so
\[
\sum_{\substack{n<b^k\\ s_b(n)\equiv a\Mod{m}}}\Lambda(n)=\frac{\sum_{n<b^k}\Lambda(n)}{m}+\sum_{\substack{\alpha=j/m\\ 1\le j\le m-1}}\frac{e(-a\alpha)}{mb^k}\sum_{0\le a<b^k}\widehat{g}_{b^k}\Bigl(\frac{a}{b^k}\Bigr)\widehat{\Lambda}_{b^k}\Bigl(-\frac{a}{b^k}\Bigr).
\]
For a general set $\mathcal{A}$ one cannot hope to prove pointwise bounds better than $\widehat{\mathbf{1}}_{\mathcal{A},x}(\theta)\ll x^{1/2}$ (so-called `square-root cancellation') since even if $\mathbf{1}_{\mathcal{A}}$ and $e(n\theta)$ behaved as completely independent random variables taking values on the complex unit circle, we would expect the sum to typically be of size $x^{1/2}$. This would mean that it would be impossible to show a bound better than $O(b^k)$ for the contribution of the minor arcs, and so an application of the circle method would be doomed to failure. (This issue, which affects all binary problems, is one reason why the circle method is not effective in the twin prime of Goldbach conjectures, for example.) It turns out that the Fourier transforms of digital functions are much more structured, and so typically have \textit{better} than square-root cancellation. This allows us to succeed in using the circle method to these `binary' problems.
\section{Exponential sums for primes}\label{sec:ExponentialSums}
We first establish our key result for exponential sums over primes. This shows that whenever $\theta$ is far from a rational with small denominator, $\widehat{\Lambda}(\theta)$ is small. The bounds here are well-known, based on the original work of Vinogradov but we give a complete proofs for completeness following the approach of Vaughan (see, for example \cite[Chapter 25]{Davenport}).
\begin{lmm}\label{lmm:Equidistribution}
Let $\alpha=a/d+\beta$ with $a,d$ coprime integers and $\beta\in\mathbb{R}$ satisfying $|\beta|<1/d^2$. Then we have
\[\sum_{n=1}^N\min\Bigl(M,\|\alpha n\|^{-1}\Bigr)\ll \Bigl(N+N M d|\beta|+\frac{1}{d|\beta|}+d\Bigr)\log{N}.\]
\end{lmm}
\begin{proof}
If $Nd|\beta|<1/2$ then we let $n=n_0+d n_1$ for non-negative integers $n_0,n_1$ with $n_0<d$ and $n_1<N/d$. If $n_0\ne 0$ then 
\[\|\alpha n\|=\|n_0a/d+ \beta n\|\ge \|n_0a/d\|-\|\beta n\|\ge \|n_0a/d\|/2\]
since $N|\beta|<1/2d$. We let $b\in\{0,\dots,d-1\}$ be such that $b\equiv n_0a\Mod{d}$. Thus the terms with $n_0\ne 0$ contribute a total
\[\ll \sum_{n_1<N/d}\sum_{1\le b<\min(d,N)}\frac{d}{b}\ll \sum_{n_1<N/d}d\log{N}\ll (N+d)\log{N}.\]
The terms with $n_0=0$ contribute
\[\ll \sum_{1\le n_1<N/d}\min\Bigl(M,\|d n_1\beta\|^{-1}\Bigr)\ll \sum_{1\le n_1<N/d}\frac{1}{d n_1 |\beta|}\ll \frac{\log{N}}{d|\beta|}.\]
Here we have used the fact that since $n_0=0$ and we sum over $n\ge 1$ we must have $n_1\ge 1$.

We now consider the case $Nd|\beta|>1/2$. We let $n=n_0+d n_1+d\lfloor(d^2\beta)^{-1}\rfloor n_2$, with $0\le n_0<d$, $0\le n_1\le (d^2\beta)^{-1}$ and $0\le n_2\ll N d \beta$. Thus we obtain
\[\sum_{n=1}^N\min\Bigl(M, \|\alpha n\|^{-1}\Bigr)\ll\sum_{\substack{n_1\le 1/d^2\beta\\ n_2\ll Nd\beta}}\sum_{0\le n_0<d}\min\Bigl(M,\Bigl\|\theta+n_1d\beta+n_0(a/d+\beta)\Bigr\|^{-1}\Bigr)\]
where we have put $\theta=\beta d \lfloor(d^2\beta)^{-1}\rfloor n_2$ for convenience. The inner sum is of the form $\sum_i\min(M,\|\theta_i\|^{-1})$ for $d$ points $\theta_i$ which are $1/2d$ separated. Therefore the sum over $n_0$ is
\begin{align*}
&\ll d\log{d}+\sup_{0\le n_0<d}\min\Bigl(M,\|\theta+n_1d\beta+n_0(a/d+\beta)\|^{-1}\Bigr)\\
&\ll d\log{d}+\min\Bigl(M,\frac{d}{\|d\theta+(n_1+O(1))d^2\beta\|}\Bigr)
\end{align*}
since $\|t\|^{-1}\le d\|dt\|^{-1}$ for all $t$. The term $d\log{d}$ contributes $\ll N\log{N}$ to the total sum, which is acceptable. Thus we are left to bound
\[\sum_{n_2\ll Md\beta}\sup_{\substack{\theta\in\mathbb{R}}}\sum_{n_1\le 1/d^2\beta}\min\Bigl(M,\frac{d}{\|d\theta+(n_1+O(1))d^2\beta\|}\Bigr).\]
 The inner sum is of the form ($O(1)$ copies of) $\sum_i\min(M,d\|\theta_i\|^{-1})$ for $O(1/d^2\beta)$ points $\theta_i$ which are $d^2\beta$-separated $\mod{1}$. Therefore the inner sum is $\ll(M+d/d^2\beta)\log{N}$, and this gives a bound
\[\ll Nd|\beta|\Bigl(M+\frac{1}{d|\beta|}\Bigr)\log{N}\ll \Bigl(MN d|\beta|+N\Bigr)\log{N}.\]
Putting these bounds together gives
\[\sum_{n=1}^N\min\Bigl(M,\|\alpha n\|^{-1}\Bigr)\ll \Bigl(N+NM d|\beta|+\frac{1}{d|\beta|}+d\Bigr)\log{N},\]
as required.
\end{proof}
\begin{lmm}\label{lmm:PrimeSum}
Let $\alpha=a/d+\beta$ with $(a,d)=1$ and $|\beta|<1/d^2$. Then
\[\widehat{\Lambda}_x(\alpha)=\sum_{n<x}\Lambda(n)e(n\alpha)\ll \Bigl(x^{4/5}+\frac{x^{1/2}}{|d\beta|^{1/2}}+x|d\beta|^{1/2}\Bigr)(\log{x})^4.\]
\end{lmm}
\begin{proof}
From \cite[(6), Page 142]{Davenport}, taking $f(n)=e(n\alpha)$ we have that for any choice of $U,V\ge 2$ with $UV\le x$
\begin{align*}
\sum_{n<x}\Lambda(n)e(n\alpha)&\ll U+(\log{x})\sum_{1\le t<UV}\sup_w\Bigl|\sum_{w<r\le x/t}e(rt\alpha)\Bigr|\\
&+x^{1/2}(\log{x})^3\sup_{\substack{U\le M\le x/V\\ V\le j\le x/M}}\Bigl(\sum_{V<k<x/M}\Bigl|\sum_{\substack{M<m\le 2M\\ m\le x/k\\ m\le x/j}}e(\alpha m (j-k))\Bigr|\Bigr)^{1/2}.
\end{align*}
The sum over $r$ is clearly $\ll \min(x/t,\|t\alpha\|^{-1})$ and the sum over $m$ is similarly $\ll \min(M,\|(j-k)\alpha\|^{-1})$. Putting $t$ and $j-k$ into dyadic intervals and applying Lemma \ref{lmm:Equidistribution} to the resulting sums (or the trival bound when $j=k$) gives a bound
\[\ll \Bigl(UV+xd|\beta|+\frac{1}{d|\beta|}+d+\frac{x}{U^{1/2}}+\frac{x}{V^{1/2}}+x|d\beta|^{1/2}+\frac{x^{1/2}}{|d\beta|^{1/2}}+x^{1/2}d^{1/2}\Bigr)(\log{x})^4.\]
Choosing $U=V=x^{2/5}$ and simplifying the terms then gives the result.
\end{proof}
\section{\texorpdfstring{$L^1$ and $L^\infty$ Fourier bounds for digit functions}{L1 and L infinity Fourier bounds for digit functions}}\label{sec:Fourier}
We now establish in turn several properties of the functions $\widehat{\mathbf{1}}_{\mathcal{B},b^k}$ and $\widehat{g}_{b^k}$, which are the key ingredients in our results. The key property is that although either of these Fourier transforms can occasionally be large, they are typically very small. In particular, if we average over a set which doesn't have particular digital structure (such as all points in an interval, or approximations to rationals) then we obtain good bounds, which show better than square-root cancellation on average provided that the base $b$ is large enough. Estimates of these types are the standard means of gaining control over digit-related functions. Lemmas \ref{lmm:L1Bound} and \ref{lmm:TypeI} should be compared with \cite[Th\'eor\`eme 2]{FouvryMauduit}, \cite[Th\'eor\`eme 2]{DartygeMauduit}, \cite[Lemme 16]{MauduitRivat} or \cite{Konyagin}, whilst Lemma \ref{lmm:LInfBound} is essentially given by \cite[Section 3]{ErdosMauduitSarkozy} or \cite[Theorem 2]{MauduitSarkozy}. 
\begin{lmm}[$L^1$ bounds]\label{lmm:L1Bound}
There exists a constant $C\ll 1$ such that
\[\sup_{\theta\in\mathbb{R}}\sum_{0\le a< b^k}\Bigl|\widehat{\mathbf{1}}_{\mathcal{B},b^k}\Bigl(\theta+\frac{a}{b^k}\Bigr)\Bigr|\ll (C b\log{b})^k,\]
and
\[
\sup_{\theta\in\mathbb{R}}\sum_{0\le a< b^k}\Bigl|\widehat{g}_{b^k}\Bigl(\theta+\frac{a}{b^k}\Bigr)\Bigr|\ll (C b\log{b})^k.
\]
\end{lmm}
We see that the average size of $\widehat{\mathbf{1}}_{\mathcal{B},b^k}$ and $\widehat{g}_{b^k}$ is at most $(C\log{b})^k$. For $b$ large, this is a very small power of $b^k$, and so the typical size is significantly smaller that $b^{k/2}$, which would correspond to square-root cancellation.
\begin{proof}
We recall the product formulae for $\widehat{\mathbf{1}}_{\mathcal{B}}$ and $\hat{g}$:
\begin{align*}
\widehat{\mathbf{1}}_{\mathcal{B},b^k}(t)&=\prod_{i=0}^{k-1}\Bigl(\frac{e(b^{i+1}t)-1}{e(b^{i} t)-1}-e(a_0 b^{i} t)\Bigr),\\
\widehat{g}_{b^k}(t)&=\prod_{i=0}^{k-1}\Bigl(\frac{e(b^{i+1}t+b\alpha)-1}{e(b^{i} t+\alpha)-1}\Bigr).
\end{align*}
The terms in parentheses can be bounded by
\begin{align}
\Bigl|\frac{e(b^{i+1}t)-1}{e(b^{i}t)-1}-e(a_0 b^{i} t)\Bigr|\le \min\Bigl(b,1+\frac{1}{\|b^{i} t\|}\Bigr),\label{eq:LittleSum}\\
\Bigl|\frac{e(b^{i+1}t+b\alpha)-1}{e(b^{i} t+\alpha)-1}\Bigr|\le \min\Bigl(b,\frac{1}{\|b^{i}t+\alpha\|}\Bigr).
\end{align}
For $t\in[0,1)$, we write $t=\sum_{i=1}^{k} t_i b^{-i}+\epsilon$ with $t_1,\dots,t_k\in\{0,\dots,b-1\}$ and $\epsilon\in[0,1/b^k)$. We see that $\|b^i t\|^{-1}=\|t_{i+1}/b+\epsilon_i\|^{-1}$ for some $\epsilon_i\in[0,1/b)$. In particular, $\|b^i t\|^{-1}\le b/t_{i+1}+b/(b-1-t_{i+1})$ if $t_{i+1}\ne 0,b-1$. Thus we see that
\begin{align*}
\sup_{\theta\in\mathbb{R}}\sum_{0\le a <b^k}\Bigl|\widehat{\mathbf{1}}_{\mathcal{B},b^k}\Bigl(\theta+\frac{a}{b^k}\Bigr)\Bigr|&\ll \sum_{t_1,\dots,t_k<b}\prod_{i=1}^k\min\Bigl(b,1+\frac{b}{t_i}+\frac{b}{b-1-t_i}\Bigr)\\
&\ll \prod_{i=1}^k \Bigl(3b+2\sum_{1\le t_i\le (b-1)/2}\frac{b}{t_i}\Bigr)\\
&\ll (C b\log{b})^k,
\end{align*}
and that exactly the same argument applies for bounding $\hat{g}_{b^k}$.
\end{proof}
\begin{lmm}[Large sieve estimates]\label{lmm:TypeI}
There is a constant $C\ll1$ such that
\begin{align*}
\sup_{\theta\in\mathbb{R}}\sum_{d\sim D}\sum_{\substack{0<\ell<d\\ (\ell,d)=1}}\sup_{|\epsilon|<\frac{1}{10D^2}}\Bigl|\widehat{\mathbf{1}}_{\mathcal{B},b^k}\Bigl(\frac{\ell}{d}+\theta+\epsilon\Bigr)\Bigr|&\ll (D^2+b^k)(C\log{b})^k,\\
\sup_{\theta\in\mathbb{R}}\sum_{d\sim D}\sum_{\substack{0<\ell<d\\ (\ell,d)=1}}\sup_{|\epsilon|<\frac{1}{10D^2}}\Bigl|\widehat{g}_{b^k}\Bigl(\frac{\ell}{d}+\theta+\epsilon\Bigr)\Bigr|&\ll (D^2+b^k)(C\log{b})^k.
\end{align*}
\end{lmm}
Again, we see that since we are summing over $D^2$ terms, if $D^2>b^k$ then on average our Fourier transforms are of size $O(b^{\delta k})$ if $b$ is large enough compared with $\delta>0$. 
\begin{proof}
The key point is that the $D^2$ points $\ell/d$ are all fairly evenly spaced in the interval $[0,1]$, so we may approximate this sum by an integral with an error which is controlled by the derivative, a formulation of the large sieve going back to Gallagher \cite{Gallagher}.

We have that for any smooth function $F$
\[F(t)=F(u)-\int_t^u F'(v)dv.\]
Thus integrating over $u\in [t-\delta,t+\delta]$ we have
\[|F(t)|\ll \frac{1}{\delta}\int_{t-\delta}^{t+\delta}|F(u)|du+\int_{t-\delta}^{t+\delta}|F'(u)|du.\]
We note that the fractions $\ell/d+\theta+\epsilon$ with $(\ell,d)=1$, $d<2D$ and $|\epsilon|<1/10D^2$ are separated from one another by $\gg 1/D^2$ . Thus
\[\sum_{d\sim D}\sum_{\substack{0<\ell<d\\ (\ell,d)=1}}\sup_{|\epsilon|<1/10D^2}\Bigl|F\Bigl(\frac{\ell}{d}+\theta+\epsilon\Bigr)\Bigr|\ll D^2\int_0^1|F(u)|du+\int_{0}^{1}|F'(u)|du.\]
We will now apply this with $F=\widehat{g}$ and $F=\widehat{\mathbf{1}}_{\mathcal{B}}$.

We note that, writing $n=\sum_{i=0}^{k-1}n_i b^i$ we have
\begin{align*}
\widehat{\mathbf{1}}_{\mathcal{B},b^k}'(t)&=2\pi i\sum_{n< b^k}n\mathbf{1}_{\mathcal{B}}(n)e(n t)\\
&=2\pi i \sum_{j=0}^{k-1}b^j\Bigl(\sum_{0\le n_j<b}n_j\mathbf{1}_{\mathcal{B}}(n_j)e(n_jb^{j}t)\Bigr)\prod_{i\ne j}\Bigl(\sum_{0\le n_i<b}\mathbf{1}_{\mathcal{B}}(n_i)e(n_ib^it)\Bigr).
\end{align*}
Thus, as in Lemma \ref{lmm:L1Bound}, we have
\[|\widehat{\mathbf{1}}_{\mathcal{B}, b^k}'(t)|\ll \sum_{j=0}^{k-1}b^{j+1}\prod_{i\ne j}\min\Bigl(b,1+\frac{1}{\|b^it\|}\Bigr)\ll b^k\prod_{i=0}^{k-1}\min\Bigl(b,1+\frac{1}{\|b^it\|}\Bigr),\]
and we have the same bound for $|\widehat{\mathbf{1}}_{\mathcal{B},b^k}(t)|$ but without the $b^k$ factor. We let $t=\sum_{j=1}^kt_j b^{-j}+\epsilon$ for some $t_1,\dots,t_k\in\{0,\dots,b-1\}$ and $\epsilon\in[0,1/b^k)$. We see that, as in Lemma \ref{lmm:L1Bound} we have
\begin{align*}
\int_0^1\prod_{i=0}^{k-1}\min\Bigl(b,1+\frac{1}{2\|b^it\|}\Bigr)dt&\ll \frac{1}{b^k}\sum_{t_1,\dots,t_k<b}\prod_{i=0}^{k-1}\Bigl(1+\min\Bigl(b,\frac{b}{2t_i},\frac{b}{2(b-1-t_i)}\Bigr)\Bigr)\\
&\ll (C\log{b})^k.
\end{align*}
Putting this all together then gives the result for $\widehat{\mathbf{1}}_{\mathcal{B},b^k}$. The result for $\hat{g}$ is entirely analogous, since
\begin{align*}
\hat{g}_{b^k}'(t)&=2\pi i \sum_{n<b^k}n e(nt+\alpha s_b(n))\\
&=2\pi i \sum_{j=0}^{k-1}b^j\Bigl(\sum_{n_j<b}n_j e((\alpha+b^j)n_j)\Bigr)\prod_{i\ne j}\Bigl(\sum_{n_i<b}e((\alpha+b^i)n_i)\Bigr),
\end{align*}
and so the same bounds apply.
\end{proof}
\begin{lmm}[Hybrid estimates]\label{lmm:ExtendedTypeI}
Let $B,D\ge 1$. There is a constant $C\ll 1$ such that
\begin{align*}
\sum_{d\sim D}\sum_{\substack{\ell<d\\ (\ell,d)=1}}\sum_{\substack{|\eta|<B\\ b^k\ell/d+\eta\in\mathbb{Z}}}\Bigl|\widehat{\mathbf{1}}_{\mathcal{B},b^k}\Bigl(\frac{\ell}{d}+\frac{\eta}{b^k}\Bigr)\Bigr|&\ll (b-1)^k(D^2B)^{\alpha_b}+D^2B(C\log{b})^k,\\
\sum_{d\sim D}\sum_{\substack{\ell<d\\ (\ell,d)=1}}\sum_{\substack{|\eta|<B\\ b^k\ell/d+\eta\in\mathbb{Z}}}\Bigl|\widehat{g}_{b^k}\Bigl(\frac{\ell}{d}+\frac{\eta}{b^k}\Bigr)\Bigr|&\ll b^k (D^2B)^{\beta_b}+D^2B(C\log{b})^k,
\end{align*}
where
\[\alpha_b=\frac{\log\Bigl(C\frac{b}{b-1}\log{b}\Bigr)}{\log{b}},\qquad \beta_b=\frac{\log\Bigl(C\log{b}\Bigr)}{\log{b}}.\]
\end{lmm}
\begin{proof}
The key idea here is to use our product formula to essentially factorize the double summation into the product of two single summations, since only certain terms in the product really depend on $\eta$.

The result follows immediately from Lemma \ref{lmm:L1Bound} if $B\ge b^k$, so we may assume $B<b^k$. For any integer $k_1\in [0,k]$ we have
\begin{align*}
\widehat{\mathbf{1}}_{\mathcal{B},b^k}(\alpha)&=\prod_{i=0}^{k-k_1-1}\Bigl(\sum_{n_i<b}\mathbf{1}_{\mathcal{B}}(n_i)e(n_ib^i\alpha)\Bigr)\prod_{i=k-k_1}^{k-1}\Bigl(\sum_{n_i<b}\mathbf{1}_{\mathcal{B}}(n_i)e(n_ib^i\alpha)\Bigr)\\
&=\widehat{\mathbf{1}}_{\mathcal{B},b^{k-k_1}}(\alpha)\widehat{\mathbf{1}}_{\mathcal{B},b^{k_1}}(b^{k-k_1}\alpha).
\end{align*}
Using this and the trivial bound $|\widehat{\mathbf{1}}_{\mathcal{B},b^j}(\theta)|\le (b-1)^j$,  for $k_1+k_2\le k$ we have that
\[\Bigl|\widehat{\mathbf{1}}_{\mathcal{B},b^k}\Bigl(\frac{\ell}{d}+\frac{\eta}{b^k}\Bigr)\Bigr|\le (b-1)^{k-k_1-k_2}\Bigl|\widehat{\mathbf{1}}_{\mathcal{B},b^{k_1}}\Bigl(\frac{b^{k-k_1}\ell}{d}+\frac{\eta}{b^{k_1}}\Bigr)\Bigr|\sup_{|\epsilon|\le B/b^k}\Bigl|\widehat{\mathbf{1}}_{\mathcal{B},b^{k_2}}\Bigl(\frac{\ell}{d}+\epsilon\Bigr)\Bigr|.\]
Substituting this bound gives
\begin{align*}
\sum_{d\sim D}&\sum_{\substack{\ell<d\\ (\ell,d)=1}}\sum_{\substack{|\eta|<B\\ b^k\ell/d+\eta\in\mathbb{Z}}}\Bigl|\widehat{\mathbf{1}}_{\mathcal{B},b^k}\Bigl(\frac{\ell}{d}+\frac{\eta}{b^k}\Bigr)\Bigr|\ll (b-1)^{k-k_1-k_2}\\
&\times\sum_{d\sim D}\sum_{\substack{\ell<d\\ (\ell,d)=1}}\sup_{|\epsilon|<B/b^k}\Bigl|\widehat{\mathbf{1}}_{\mathcal{B},b^{k_2}}\Bigl(\frac{\ell}{d}+\epsilon\Bigr)\Bigr|\sum_{\substack{|\eta|<B\\ b^k\ell/d+\eta\in\mathbb{Z}}}\Bigl|\widehat{\mathbf{1}}_{\mathcal{B},b^{k_1}}\Bigl(\frac{b^{k-k_1}\ell}{d}+\frac{\eta}{b^{k_1}}\Bigr)\Bigr|.
\end{align*}
We choose $k_1$ minimally such that $b^{k_1}>B$, and extend the inner sum to $|\eta|<b^{k_1}$. Applying Lemma \ref{lmm:L1Bound} to the inner sum, and then Lemma \ref{lmm:TypeI} to the sum over $d,\ell$ gives
\[\sum_{d\sim D}\sum_{\substack{\ell<d\\ (\ell,d)=1}}\sum_{\substack{|\eta|<B\\ b^k\ell/d+\eta\in\mathbb{Z}}}\Bigl|\widehat{\mathbf{1}}_{\mathcal{B},b^k}\Bigl(\frac{\ell}{d}+\frac{\eta}{b^k}\Bigr)\Bigr|\ll (b-1)^{k-k_1-k_2}b^{k_1}(b^{k_2}+D^2)(C\log{b})^{k_1+k_2}.\]
We choose $k_2=\min(k-k_1,\lfloor2\log{D}/\log{b}\rfloor)$. We see that
\begin{align*}
\Bigl(\frac{C b\log{b}}{b-1}\Bigr)^{k_1+k_2}&\ll (D^2B)^{\alpha_b},\\
D^2 b^{k_1}\Bigl(\frac{C\log{b}}{b-1}\Bigr)^{k_1+k_2}&\ll \frac{D^2 B}{(b-1)^k}(C\log{b})^k+(D^2B)^{\alpha_b}.
\end{align*}
Combining these bounds gives the result for $\widehat{\mathbf{1}}_{\mathcal{B},b^k}$.

We can apply an identical argument to $\widehat{g}$. The only difference is that our trivial bound is $\widehat{g}(\theta)\ll b^k$, so we obtain
\[\sum_{d\sim D}\sum_{\substack{\ell<d\\ (\ell,d)=1}}\sum_{\substack{|\eta|<B\\ b^k\ell/d+\eta\in\mathbb{Z}}}\Bigl|\widehat{\mathbf{1}}_{\mathcal{B},b^k}\Bigl(\frac{\ell}{d}+\frac{\eta}{b^k}\Bigr)\Bigr|\ll b^{k-k_1-k_2}b^{k_1}(b^{k_2}+D^2)(C\log{b})^{k_1+k_2}.\]
Choosing $k_1,k_2$ as above, this gives the analogous result for $\widehat{g}$ with $b-1$ replaced by $b$ and $\alpha_b$ replaced by $\beta_b$.
\end{proof}
\begin{lmm}[$L^\infty$ bounds]\label{lmm:LInfBound}\hfill\\
\begin{enumerate}
\item Let $d<b^{k/3}$ be of the form $d=d_1d_2$ with $(d_1,b)=1$ and $d_1\ne 1$, and let $|\epsilon|<1/2b^{2k/3}$. Then for any integer $\ell$ coprime with $d$ we have 
\[\Bigl| \widehat{\mathbf{1}}_{\mathcal{B},b^k}\Bigl(\frac{\ell}{d}+\epsilon\Bigr)\Bigr|\ll (b-1)^{k}\exp\Bigl(-c_b\frac{k}{\log{d}}\Bigr)\]
for some constant $c_b>0$ depending only on $b$.

\item For any $\theta\in\mathbb{R}$, we have
\[|\hat{g}(\theta)|\ll b^{k(1-\delta_{\alpha,b})},\]
where $\delta_{\alpha,b}=\|(b-1)\alpha\|^2/(4b^4)$.
\end{enumerate}
\end{lmm}
This shows that $\widehat{g}(\theta)$ is uniformly small for $\alpha=j/m$ with $1\le j\le m-1$ and $(m,b-1)=1$, and that $\widehat{\mathbf{1}}_{\mathcal{B},b^k}(\theta)$ is small whenever $\theta$ is close to a rational with small denominator which is not a divisor of $b^k$. 
\begin{proof}
We have that
\[|e(n\theta)+e((n+1)\theta)|^2=2+2\cos(2\pi \theta)<4\exp(-2\|\theta\|^2).\]
This implies that
\[\Bigl|\sum_{n_i<b}\mathbf{1}_{\mathcal{B}}(n_i)e(n_i \theta)\Bigr|\le b-3+2\exp(-\|\theta\|^2)\le (b-1)\exp\Bigl(-\frac{\|\theta\|^2}{b}\Bigr).\]
We substitute this bound into our expression for $\widehat{\mathbf{1}}_{\mathcal{B}}$, which gives
\begin{align*}
\Bigl|\widehat{\mathbf{1}}_{\mathcal{B},b^k}(t)\Bigr|&=\prod_{i=0}^{k-1}\Bigl|\sum_{n_i<b}\mathbf{1}_{\mathcal{B}}(n_i)e(n_ib^it)\Bigr|\\
&\le (b-1)^{k}\exp\Bigl(-\frac{1}{b}\sum_{i=0}^{k-1}\|b^it\|^2\Bigr).
\end{align*}
If $\|b^it\|<1/2b$ then $\|b^{i+1}t\|=b\|b^i t\|$. If $t=\ell/d_1d_2$ with $d_1>1$, $(d_1,b)=1$ and $(\ell,d_1)=1$ then $\|b^it\|\ge 1/d_1d_2$ for all $i$. Similarly, if $t=\ell/d_1d_2+\epsilon$ with $\ell,d_1,d_2$ as above $|\epsilon|<b^{-2k/3}/2$ and $d=d_1d_2<b^{k/3}$ then for $i<k/3$ we have $\|b^it\|\ge 1/d-b^i|\epsilon|\ge 1/2d$. Thus, for any interval $\mathcal{I}\subseteq[0,k/3]$ of length $\log{d}/\log{b}$, there must be some integer $i\in\mathcal{I}$ such that $\|b^i(\ell/d+\epsilon)\|>1/2b^2$. This implies that
\[\sum_{i=0}^k\Bigl\|b^i\Bigl(\frac{\ell}{d}+\epsilon\Bigr)\Bigr\|^2\ge \frac{1}{4b^4}\Bigl\lfloor\frac{k\log{b}}{3\log{d}}\Bigr\rfloor.\]
Substituting this into the bound for $\widehat{\mathbf{1}}_{\mathcal{B},b^k}$, and recalling we assume $d<b^{k/3}$ gives the first result.

For $\widehat{g}$, we notice that since $\|t\|\ge \|bt\|/b$ and $\|t_1\|+\|t_2\|\ge \|t_1-t_2\|$, we have
\begin{align*}
\|b^i\theta+\alpha\|+\|b^{i+1}\theta+\alpha\|&\ge \frac{1}{b}\Bigl(\|b^{i+1}\theta+b\alpha\|+\|b^{i+1}\theta+\alpha\|\Bigr)\\
&\ge \frac{\|(b-1)\alpha\|}{b}.
\end{align*}
Thus
\[
\sum_{i=0}^{k-1}\|b^i\theta+\alpha\|^2\ge \frac{1}{4}\sum_{i=0}^{k-2}\Bigl(\|b^i\theta+\alpha\|+\|b^{i+1}\theta+\alpha\|\Bigr)^2\ge \frac{k}{4 b^2}\|(b-1)\alpha\|^2.
\]
Using this in an analogous argument to the one for $\widehat{\mathbf{1}}_{\mathcal{B},b^k}$, we find
\begin{align*}
|\widehat{g}_{b^k}(\theta)|&\le \prod_{i=0}^{k-1}\Bigl|\sum_{n_i<b}e(n_i(b^i+\alpha)\Bigr|\\
&\le b^k\exp\Bigl(\frac{-1}{b}\sum_{i=0}^{k-1}\|b^i\theta+\alpha\|^2\Bigr)\\
&\le b^k \exp\Bigl(\frac{k\|(b-1)\alpha\|^2}{4b^3}\Bigr).
\end{align*}
This gives the second result on noting that $\log{b}\le b$.
\end{proof}
\section{Minor arcs}
We now use the exponential sum estimates from the previous sections to show that when $\alpha$ is `far' from a rational with small denominator the quantities $\widehat{\mathbf{1}}_{\mathcal{B},b^k}(\alpha)\hat{\Lambda}_{b^k}(-\alpha)$ and $\widehat{g}(\alpha)\hat{\Lambda}_{b^k}(-\alpha)$ are typically small in absolute value.
\begin{lmm}\label{lmm:PrimeMinor}
Let $1\ll B\ll b^k/D_0D$ and $1\ll D\ll D_0\ll b^{k/2}$. Then we have
\begin{align*}
\sum_{d\sim D}\sum_{\substack{0<\ell<d\\ (\ell,d)=1}}\sum_{\substack{|\eta|\sim B\\ b^k\ell/d+\eta\in\mathbb{Z}}}\Bigl|\widehat{\mathbf{1}}_{\mathcal{B},b^k}\Bigl(\frac{\ell}{d}+\frac{\eta}{b^k}\Bigr)\widehat{\Lambda}_{b^k}\Bigl(-\frac{\ell}{d}-\frac{\eta}{b^k}\Bigr)\Bigr|
&\ll (b-1)^kb^{k}\Bigl(\frac{k^4}{(DB)^{1/5-\alpha_b}}+\frac{k^4b^{k\alpha_b}}{D_0^{1/2}}\Bigr),\\
\sum_{d\sim D}\sum_{\substack{0<\ell<d\\ (\ell,d)=1}}\sum_{\substack{|\eta|\ll 1\\ b^k\ell/d+\eta\in\mathbb{Z}}}\Bigl|\widehat{\mathbf{1}}_{\mathcal{B},b^k}\Bigl(\frac{\ell}{d}+\frac{\eta}{b^k}\Bigr)\widehat{\Lambda}_{b^k}\Bigl(-\frac{\ell}{d}-\frac{\eta}{b^k}\Bigr)\Bigr|
&\ll  (b-1)^kb^{k}\Bigl(\frac{k^4}{D^{1/5-\alpha_b}}+\frac{k^4D_0^{1/2+2\alpha_b}}{b^{k/2}}\Bigr),\\
\sum_{d\sim D}\sum_{\substack{0<\ell<d\\ (\ell,d)=1}}\sum_{\substack{|\eta|\sim B\\ b^k\ell/d+\eta\in\mathbb{Z}}}\Bigl|\widehat{g}_{b^k}\Bigl(\frac{\ell}{d}+\frac{\eta}{b^k}\Bigr)\widehat{\Lambda}_{b^k}\Bigl(-\frac{\ell}{d}-\frac{\eta}{b^k}\Bigr)\Bigr|
&\ll  b^{2k}\Bigl(\frac{k^4}{(DB)^{1/5-\beta_b}}+\frac{k^4b^{k\beta_b}}{D_0^{1/2}}\Bigr),\\
\sum_{d\sim D}\sum_{\substack{0<\ell<d\\ (\ell,d)=1}}\sum_{\substack{|\eta|\ll 1\\ b^k\ell/d+\eta\in\mathbb{Z}}}\Bigl|\widehat{g}_{b^k}\Bigl(\frac{\ell}{d}+\frac{\eta}{b^k}\Bigr)\widehat{\Lambda}_{b^k}\Bigl(-\frac{\ell}{d}-\frac{\eta}{b^k}\Bigr)\Bigr|
&\ll  b^{2k}\Bigl(\frac{k^4}{D^{1/5-\beta_b}}+\frac{k^4D_0^{1/2+2\beta_b}}{b^{k/2}}\Bigr).
\end{align*}
Here $\alpha_b$ and $\beta_b$ are the constants described in Lemma \ref{lmm:ExtendedTypeI}.
\end{lmm}
The key point in this result is that provided $\alpha_b<1/5-\epsilon$ and $D_0^{1/2+2\alpha_b}<b^{k/2-k\epsilon}$ and $D_0^{1/2}>b^{k(\alpha_b-\epsilon)}$, we show that even if we sum over all fractions which are close to rationals with denomitators of size $D$, we have a power saving over the trivial bound $(b-1)^kb^k$ of $\widehat{\mathbf{1}}_{\mathcal{B},b^k}(\theta)\widehat{\Lambda}(-\theta)$, and these terms can be ignored in our application of the circle method. Since $\alpha_b\rightarrow 0$ as $b\rightarrow \infty$ these inequlaties can all be satisfied on choosing $D_0=b^{k/2}$ and taking $b$ sufficiently large.
\begin{proof}
By Lemma \ref{lmm:ExtendedTypeI} we have that if $D^2B\ll b^k$ then
\[\sum_{d\sim D}\sum_{\substack{\ell<d\\ (\ell,d)=1}}\sum_{\substack{|\eta|<B\\ b^k\ell/d+\eta\in\mathbb{Z}}}\Bigl|\widehat{\mathbf{1}}_{\mathcal{B},b^k}\Bigl(\frac{\ell}{d}+\frac{\eta}{b^k}\Bigr)\Bigr|\ll (b-1)^k(D^2B)^{\alpha_b}.\]
By Lemma \ref{lmm:PrimeSum} we have
\[\sup_{\substack{d\sim D\\ (\ell,d)=1\\ |\eta|\sim B}}\Bigl|\sum_{n<b^k}\Lambda(n)e\Bigl(-n\Bigl(\frac{\ell}{d}+\frac{\eta}{b^k}\Bigr)\Bigr)\Bigr|\ll \Bigl(b^{4k/5}+\frac{b^{k}}{(DB)^{1/2}}+(DB)^{1/2}b^{k/2}\Bigr)(k\log{b})^4.\]
Putting these together gives
\begin{align*}
&\sum_{d\sim D}\sum_{\substack{0<\ell<d\\ (\ell,d)=1}}\sum_{\substack{|\eta|\sim B\\ b^k\ell/d+\eta\in\mathbb{Z}}}\Bigl|\widehat{\mathbf{1}}_{\mathcal{B},b^k}\Bigl(\frac{\ell}{d}+\frac{\eta}{b^k}\Bigr)\widehat{\Lambda}_{b^k}\Bigl(-\frac{\ell}{d}-\frac{\eta}{b^k}\Bigr)\Bigr|\\
&\ll  k^4 b^k(b-1)^k\Bigl(\frac{(D^2B)^{\alpha_b}}{b^{k/5}}+\frac{(D^2B)^{\alpha_b}}{(DB)^{1/2}}+\frac{(DB)^{1/2}(D^2B)^{\alpha_b}}{b^{k/2}}\Bigr)(\log{b})^4.
\end{align*}
Recalling that $D^2B<b^k$ and $DB<b^k/D_0$ by assumption, we see that this is
\[\ll k^4 b^k(b-1)^k \Bigl((D^2B)^{\alpha_b-1/5}+(D^2B)^{\alpha_b-1/4}+\frac{b^{k\alpha_b}}{D_0^{1/2}}\Bigr),\]
and the first term clearly dominates the second. 

By partial summation we see that we obtain the same bound for $\widehat{\Lambda}_{b^k}(\alpha+O(1/b^k))$ as the bound for $\widehat{\Lambda}_{b^k}(\alpha)$ given in Lemma \ref{lmm:PrimeSum}. Thus in the case $|\eta|\ll 1$ we obtain the well-known bound 
\[\sup_{\substack{d\sim D\\ (\ell,d)=1\\ |\eta|\ll 1}}\Bigl|\sum_{n<b^k}\Lambda(n)e\Bigl(-n\Bigl(\frac{\ell}{d}+\frac{\eta}{b^k}\Bigr)\Bigr)\Bigr|\ll \Bigl(b^{4k/5}+\frac{b^{k}}{D^{1/2}}+D^{1/2}b^{k/2}\Bigr)(k\log{b})^4.\]
This gives
\begin{align*}
&\sum_{d\sim D}\sum_{\substack{0<\ell<d\\ (\ell,d)=1}}\sum_{\substack{|\eta|\ll 1\\ b^k\ell/d+\eta\in\mathbb{Z}}}\Bigl|\widehat{\mathbf{1}}_{\mathcal{B},b^k}\Bigl(\frac{\ell}{d}+\frac{\eta}{b^k}\Bigr)\widehat{\Lambda}_{b^k}\Bigl(-\frac{\ell}{d}-\frac{\eta}{b^k}\Bigr)\Bigr|\\
&\ll  k^4 b^k(b-1)^k\Bigl(\frac{D^{2\alpha_b}}{b^{k/5}}+\frac{D^{2\alpha_b}}{D^{1/2}}+\frac{D^{1/2+2\alpha_b}}{b^{k/2}}\Bigr).
\end{align*}
Recalling that $1\ll D\ll D_0\ll b^{k/2}$ then gives the result for $\widehat{\mathbf{1}}_{\mathcal{B},b^k}$. The bounds for $\widehat{g}_{b^k}$ follow an identical argument with $\alpha_b$ replaced by $\beta_b$ and occurrences of $b-1$ replaced by $b$ by using the second bound from Lemma \ref{lmm:ExtendedTypeI}.
\end{proof}

\section{Major Arcs}
We now consider $\widehat{\mathbf{1}}_{\mathcal{B},b^k}(\alpha)\widehat{\Lambda}_{b^k}(-\alpha)$ and $\widehat{g}_{b^k}(\alpha)\widehat{\Lambda}_{b^k}(-\alpha)$ when $\alpha$ is close to a rational with small denominator. By Lemma \ref{lmm:LInfBound}, $\widehat{\mathbf{1}}_{\mathcal{B},b^k}$ is small unless the denominator is a divisor of $b^k$, and $\widehat{g}_{b^k}$ is always small for such $\alpha$. Since there are very few such $\alpha$, this means they make a negligible contribution.
\begin{lmm}\label{lmm:MajorError}
If $D$, $B\ll \exp(c_b^{1/2} k^{1/2}/3)$, then we have
\[\sum_{\substack{d<D\\ \exists p|d,p\nmid b}}\sum_{\substack{0<\ell<d\\ (\ell,d)=1}}\sum_{\substack{|\eta|\ll B\\ b^k\ell/d+\eta\in\mathbb{Z}}}\Bigl|\widehat{\mathbf{1}}_{\mathcal{B},b^k}\Bigl(\frac{\ell}{d}+\frac{\eta}{b^k}\Bigr)\widehat{\Lambda}_{b^k}\Bigl(\frac{-\ell}{d}+\frac{-\eta}{b^k}\Bigr)\Bigr|\ll \frac{b^k(b-1)^k}{\exp(c_b^{1/2}k^{1/2})}.\]
For any $D,B\ge 1$ we have
\[\sum_{\substack{d<D\\ \exists p|d,p\nmid b}}\sum_{\substack{0<\ell<d\\ (\ell,d)=1}}\sum_{\substack{|\eta|\ll B\\ b^k\ell/d+\eta\in\mathbb{Z}}}\Bigl|\widehat{g}_{b^k}\Bigl(\frac{\ell}{d}+\frac{\eta}{b^k}\Bigr)\widehat{\Lambda}_{b^k}\Bigl(\frac{-\ell}{d}+\frac{-\eta}{b^k}\Bigr)\Bigr|\ll D^2 B b^{k(2-\delta_{\alpha,b})}.\]
Here $c_b$ and $\delta_{\alpha,b}$ are the quantities from Lemma \ref{lmm:LInfBound}.
\end{lmm}
\begin{proof}
These bounds follow immediately from Lemma \ref{lmm:LInfBound}, using the trivial bound for the exponential sum involving primes.
\end{proof}
\begin{lmm}\label{lmm:PrimeMajor}
Let $A>0$. Then for $D,E<(\log{b^k})^{A}$ and $D>b$ we have
\begin{align*}
\frac{1}{b^k}\sum_{\substack{d<D\\ p|d\Rightarrow p|b}}\sum_{\substack{0\le \ell<d\\ (\ell,d)=1}}\sum_{|e|<E}\widehat{\mathbf{1}}_{\mathcal{B},b^k}\Bigl(\frac{\ell}{d}+\frac{e}{b^k}\Bigr)&\widehat{\Lambda}_{b^k}\Bigl(\frac{-\ell}{d}+\frac{-b}{b^k}\Bigr)\\
&=\kappa_b(a_0)(b-1)^k+O_A\Bigl(\frac{(b-1)^k}{(\log{b^k})^A}\Bigr),
\end{align*}
where
\[\kappa_b(a_0)=\begin{cases}
\displaystyle\frac{b}{b-1},\qquad&\text{if $(a_0,b)\ne 1$,}\\
\displaystyle\frac{b(\phi(b)-1)}{(b-1)\phi(b)},&\text{if $(a_0,b)=1$.}
\end{cases}
\]
\end{lmm}
\begin{proof}
If $e\ne 0$ then by the prime number theorem in arithmetic progressions in short intervals and partial summation we have
\[\widehat{\Lambda}_{b^k}\Bigl(\frac{-\ell}{d}+\frac{-e}{b^k}\Bigr)\ll_A \frac{b^k}{(\log{b^k})^{4A}}.\]
Thus the terms with $e\ne 0$ contribute
\begin{align*}
\ll \frac{(\log{b^k})^{3A}}{b^k}\sup_{0<a<b^k}\Bigl|\widehat{\mathbf{1}}_{\mathcal{B},b^k}\Bigl(\frac{a}{b^k}\Bigr)\Bigr|\frac{b^k}{(\log{b^k})^{4A}}\ll \frac{(b-1)^k}{(\log{b^k})^{A}}.
\end{align*}
Here we used the trivial bound that $|\widehat{\mathbf{1}}_{\mathcal{B},b^k}(\theta)|\le (b-1)^k$ for all $\theta$.

Using the prime number theorem in arithmetic progressions again, we see that
\[\widehat{\Lambda}_{b^k}\Bigl(\frac{-\ell}{d}\Bigr)=\frac{b^k}{\phi(d)}\sum_{\substack{0<c<d\\ (c,d)=1}}e\Bigl(\frac{-l c}{d}\Bigr)+O_A\Bigl(\frac{b^k}{(\log{b^k})^{4A}}\Bigr)=\frac{\mu(d)b^k}{\phi(d)}+O_A\Bigl(\frac{b^k}{(\log{b^k})^{4A}}\Bigr).\]
Thus we may restrict to $d|b$, since all other such $d$ are not square-free. Letting $\ell'/b=\ell/d$, we see terms with $e=0$ and $d|b$ contribute
\begin{align*}
\frac{1}{b^k}\sum_{0\le \ell'<b}\widehat{\mathbf{1}}_{\mathcal{B},b^k}\Bigl(\frac{\ell'}{b}\Bigr)&\widehat{\Lambda}_{b^k}\Bigl(\frac{-\ell'}{b}\Bigr)=\frac{1}{b^{k-1}}\sum_{\substack{n,m<b^k\\ n\equiv m\Mod{b}}}\Lambda(n)\mathbf{1}_{\mathcal{B}}(m)\\
&=\frac{b}{\phi(b)}\sum_{\substack{1<a<b\\ (a,b)=1}}\sum_{\substack{m<b^k\\ m\equiv a\Mod{b}}}\mathbf{1}_{\mathcal{B}}(m)+O_A\Bigl(\frac{b^k}{(\log{b^k})^{4A}}\Bigr).
\end{align*}
If $a\ne a_0$ then the sum over $m$ is $(b-1)^{k-1}$ since there are $(b-1)$ choices for each digit of $m$ apart from the final one, which must be $a$. If $a=a_0$ then the sum is empty. Thus
\begin{align*}
&\frac{b}{\phi(b)}\sum_{\substack{1<a<b\\ (a,b)=1}}\sum_{\substack{m<b^k\\ m\equiv a\Mod{b}}}\mathbf{1}_{\mathcal{B}}(m)=\begin{cases}
b(b-1)^{k-1},\qquad &\text{if $(a_0,b)\ne 1$,}\\
\displaystyle\frac{\phi(b)-1}{\phi(b)}b(b-1)^{k-1},&\text{if $(a_0,b)=1$.}
\end{cases}
\end{align*}
Putting this together gives the result.
\end{proof}
\section{Proofs of Theorems \ref{thrm:WeakMauduit} and \ref{thrm:WeakMaynard}}

\begin{proof}[Proof of Theorem \ref{thrm:WeakMauduit}]
For Theorem \ref{thrm:WeakMauduit}, by using additive characters to detect $s_q(n)\equiv a \Mod{m}$ we have that
\begin{align*}
\sum_{\substack{n<b^k\\ s_b(n)\equiv a\Mod{m}}}\Lambda(n)&=\frac{1}{m}\sum_{n<b^k}\Lambda(n)+\frac{1}{m}\sum_{\substack{\alpha=j/m\\ 1\le j\le m-1}}e(-a\alpha)\sum_{n<b^k}\Lambda(n)g(n)\\
&=\frac{(1+o(1))b^k}{m}+O\Bigl(\sup_{\alpha=j/m}\frac{1}{b^k}\Bigl|\sum_{a<b^k}\widehat{g}_{b^k}\Bigl(\frac{a}{b^k}\Bigr)\widehat{\Lambda}\Bigl(\frac{-a}{b^k}\Bigr)\Bigr|\Bigr).
\end{align*}
Here we recall that $g(n)=e(\alpha s_b(n))$. By Dirichlet's approximation theorem, for any choice of $0<D_0$ and any $0\le a<b^k$ there exists integers $(\ell,d)=1$ with $d<D_0$ and a real $|\beta|<1/dD_0$ such that
\[\frac{a}{b^k}=\frac{\ell}{d}+\beta.\]
We see that $b^k\ell/d+b^k\beta\in\mathbb{Z}$. We choose $D_0=b^{k/2}$, and divide the sum over $a$ according to whether $\max(d,b^k|\beta|)<b^{k\delta_{\alpha,b}/4}$ or not. Those $a$ for which this bound holds can be bounded by Lemma \ref{lmm:MajorError}, whilst those for which this bound doesn't hold can be bounded by Lemma \ref{lmm:PrimeMinor}. Putting this together shows that
\begin{align*}
\sum_{\substack{n<b^k\\ s_b(n)\equiv a\Mod{m}}}\Lambda(n)
&=\frac{(1+o(1))b^k}{m}
+O\Bigl(\frac{b^{k}}{b^{k\delta_{\alpha,b}/4}}\Bigr)
+O\Bigl(\frac{k^4 b^{k}}{b^{k\delta_{\alpha,b}(1/5-\beta_b)/4}}\Bigr)\\
&\qquad +O\Bigl(\frac{b^k k^4 D_0^{1/2+2\beta_b}}{b^{k/2}}\Bigr)+O\Bigl(\frac{k^4b^{k+k\beta_b}}{D_0^{1/2}}\Bigr).
\end{align*}
Recalling that $D_0=b^{k/2}$ and that $\beta_b\rightarrow 0$ as $b\rightarrow \infty$, we see that the final two terms are negligible for $b$ sufficiently large. Recalling that by assumption $(m,b-1)=1$, we see that $\delta_{\alpha,b}=\|(b-1)\alpha\|^2/4b^4\ge 1/(4b^4m^2)$ when $\alpha=j/m$ for $1\le j\le m$. Thus for $b$ sufficiently large, first two error terms are $o(b^k)$ as well. Theorem \ref{thrm:WeakMauduit} then follows by partial summation.
\end{proof}

\begin{proof}[Proof of Theorem \ref{thrm:WeakMaynard}]
By Fourier expansion we have
\begin{align*}
\sum_{n<b^k}\Lambda(n)\mathbf{1}_{\mathcal{B}}(n)&=\frac{1}{b^k}\sum_{0\le a<b^k}\widehat{\mathbf{1}}_{\mathcal{B},b^k}\Bigl(\frac{a}{b^k}\Bigr)\widehat{\Lambda}_{b^k}\Bigl(\frac{-a}{b^k}\Bigr).
\end{align*}
We use Dirichlet's approximation theorem to write $a/b^k=\ell/d+\beta$ for some $d<D_0=b^{k/2}$ as in the proof of Theorem \ref{thrm:WeakMauduit}. We use Lemmas \ref{lmm:PrimeMajor} and \ref{lmm:MajorError} to estimate the contribution when $\max(d,b^k|\beta|)<(\log{b^k})^A$, and use Lemma \ref{lmm:PrimeMinor} for the remaining cases. This gives
\begin{align*}
&\frac{1}{b^k}\sum_{0\le a<b^k}\widehat{\mathbf{1}}_{\mathcal{B},b^k}\Bigl(\frac{a}{b^k}\Bigr)\widehat{\Lambda}_{b^k}\Bigl(\frac{-a}{b^k}\Bigr)=\kappa_b(a_0)(b-1)^k\\
&+O_A\Bigl( (b-1)^k\Bigl(\frac{1}{(\log{b^k})^A}+\frac{k^4}{(\log{b^k})^{A(1/5-\alpha_b)}}+\frac{k^5 b^{k\alpha_b}}{D_0^{1/2}}+\frac{k^5 D_0^{1/2+2\alpha_b}}{b^{k/2}}\Bigr)\Bigr).
\end{align*}
Recalling that  $D_0=b^{k/2}$ we see that the error term is $O_B((b-1)^k (\log{b^k})^{-B})$ provided $\alpha_b<1/5$ and $A$ is chosen such that $A>(B+5)/(1/5-\alpha_b)$. We recall from Lemmas \ref{lmm:ExtendedTypeI} and \ref{lmm:L1Bound} that
\[\alpha_b=\frac{\log\Bigl(C\frac{b}{b-1}\log{b}\Bigr)}{\log{b}}.\]
This clearly tends to zero as $b\rightarrow \infty$, and so in particular is less than $1/5$ for $b$ sufficiently large. Theorem \ref{thrm:WeakMaynard} then follows by partial summation.
\end{proof}

\section{Modifications for Theorem \ref{thrm:Bourgain}}
A fundamentally similar argument can also be used to prove Theorem \ref{thrm:Bourgain}. Indeed, if not too many of the prescribed digits are amongst the final digits then essentially exactly the same circle method type argument will apply. Indeed, if $\mathcal{A}$ is the set of integers with $i^{th}$ base $b$ digit is equal to $\epsilon_i$ for $i\in\mathcal{I}$, we have
\begin{align*}
\widehat{\mathbf{1}}_{\mathcal{A},b^k}(\theta)&=\sum_{n<b^k}e(n\theta)\mathbf{1}_{\mathcal{A}}(n)\\
&=\sum_{\substack{0\le n_1,\dots,n_k< b\\ n_i= \epsilon_i\,\forall i\in\mathcal{I}}}e\Bigl(\sum_{i=1}^k n_i b^{i-1}\theta\Bigr)\\
&=\prod_{i\in\mathcal{I}}e(\epsilon_i\theta)\prod_{i\notin\mathcal{I}}\Bigl(\frac{e(b^{i}\theta)-1}{e(b^{i-1}\theta)-1}\Bigr).
\end{align*}
This function enjoys many similar properties to that of $\widehat{\mathbf{1}}_{\mathcal{B},b^k}$, and in particular one can establish analogous minor arc bounds. 

Slight difficulties arise when many of the prescribed digits are amongst the final few digits. This should not be a surprise, since this is asking a generalised question of primes in arithmetic progressions, and to show the existence of primes in arithmetic progressions when the modulus is reasonably large one needs to make use of deeper facts about the density of zeros of Dirichlet $L$-functions.

If one assumes the Riemann hypothesis for Dirichlet $L$-functions, then the relevant contributions from the circle method can be handled quite easily by considering only primes in a suitable arithmetic progression (to ensure the final digits match the prescribed ones). For an unconditional result, the zero-free region for Dirichlet $L$-functions to a smooth modulus and zero density estimates serve as an adequate substitute for the Riemann Hypothesis, as is the case in many applications. For full details, we refer the reader to \cite{Bourgain}. The only care needed to generalize base 2 to general base $b$ is in section 5 of \cite{Bourgain} where Bourgain writes $q_1=2^\nu q_1'$ and splits a summation modulo $2^\nu$. Instead one can write $q_1=\tilde{q_1}q_1'$ with $(q_1',b)=1$ and $\tilde{q}_1|b^\nu$, and split the summation modulo $b^\nu$. Provided the constant $c$ of Theorem \ref{thrm:Bourgain} (which is denoted $\rho$ in \cite{Bourgain}) is sufficiently small compared with $b$, the proof then works as before, the only modification required being to replace suitable instances of 2 with $b$ and the condition `odd' with `coprime to $b$'.

\section{Acknowledgments}
The author is a fellow of Magdalen College, Oxford, and is funded by a Clay research fellowship. The thorough reading and suggestions of the anonymous referees are gratefully acknowledged.

\end{document}